\documentclass[12pt,reqno,A4]{amsart}
\usepackage{amssymb}
\usepackage{latexsym,cmtiup}

\oddsidemargin=\evensidemargin
\numberwithin{equation}{section}
\begin{document}
\newtheorem{defi}{Definition}
\newtheorem{theo}{Theorem}[section]
\newtheorem{prop}[theo]{Proposition}
\newtheorem{lemm}[theo]{Lemma}
\newtheorem{coro}[theo]{Corollary}
\newtheorem{hypo}[theo]{Hypothesis}
\newcommand{\bp}{\begin{proof}}
\newcommand{\ep}{\end{proof}}
\let\le=\leqslant
\let\ge=\geqslant
\let\leq=\leqslant
\let\geq=\geqslant
\date{}
\author{E. I. Khukhro}
\address{Charlotte Scott Research Centre for Algebra,\newline\indent University of Lincoln, Lincoln, LN6 7TS, U.K., and\newline\indent  Sobolev  Institute of Mathematics, Novosibirsk, 630090, Russia}
\email{khukhro@yahoo.co.uk}
\author{P. Shumyatsky}
\address{Department of Mathematics, University of Brasilia, DF~70910-900, Brazil}
\email{pavel@unb.br}
\author{G. Traustason}
\address{Department of Mathematical Sciences,
University of Bath, Bath
BA2 7AY, U.K.}
\email{G.Traustason@bath.ac.uk}

\title[Right Engel-type subgroups and length parameters of finite groups]{Right Engel-type subgroups\\ and length parameters of finite groups}

\maketitle

\begin{abstract}
Let $g$ be an element of a finite
group $G$ and let $R_{n}(g)$ be the subgroup generated by all the right Engel values $[g,{}_{n}x]$ over $x\in G$. In the case when
$G$ is soluble we prove that if, for some $n$, the Fitting height of $R_{n}(g)$ is equal to $k$, then $g$ belongs to the $(k+1)$th Fitting subgroup $F_{k+1}(G)$. For nonsoluble $G$, it is proved that if, for some $n$, the  generalized Fitting height of $R_n(g)$ is equal to $k$, then $g$ belongs to the generalized Fitting subgroup $F^*_{f(k,m)}(G)$ with $f(k,m)$ depending only on $k$ and $m$, where $|g|$ is the product of $m$ primes counting multiplicities. It is also proved that if, for some $n$, the nonsoluble length of   $R_n(g)$ is equal to $k$, then $g$ belongs to a normal subgroup whose nonsoluble length is bounded in terms of $k$ and $m$. Earlier similar generalizations of Baer's theorem (which states that an Engel element of a finite group belongs to the Fitting subgroup) were obtained by the first two authors in terms of left Engel-type subgroups.
\end{abstract}

\section{Introduction}

Recall that the $n$-Engel word $[y,{}_{n}x]$ is defined recursively by
$[y,{}_{0}x]=y$ and $[y,{}_{i+1}x]=[[y,{}_{i}x],x]$.
Let $G$ be a finite group.
In \cite{Khb} the authors considered the subgroups
$$
E_{n}(g)=\langle [x,{}_{n}g]:\,x\in G\rangle,
$$
where $g\in G$ and $n$ is a positive integer. By a well-known result of Baer \cite{Ba} we know that if $E_{n}(g)=1$, then $g$ belongs to the Fitting subgroup of $G$. The authors of \cite{Khb} generalized Baer's result in terms of various length parameters of $E_{n}(g)$. The aim of the present paper is to consider the right Engel analog of this setting and to prove the right Engel versions of the three main theorems in \cite{Khb}. Instead of $E_{n}(g)$ we consider
$$
R_{n}(g)=\langle [g,{}_{n}x]:\,x\in G\rangle.
$$
Before stating our first main result, we recall some terminology.
For a finite group $G$, the Fitting series is defined recursively by $F_{0}(G)=1$ and $F_{i+1}(G)/F_{i}(G)=F(G/F_{i}(G))$ where $F(H)$ is the Fitting subgroup of $H$. If $G$ is soluble, then the Fitting height $h(G)$ is the least integer $h$ such that $F_{h}(G)=G$. We can now state our first result that is about finite soluble groups.

\begin{theo}\label{t1.1}
Let $g$ be an element of a finite soluble group $G$ and $n$ a positive integer. If the Fitting height of $R_{n}(g)$ is equal to $k$, then $g$ belongs to $F_{k+1}(G)$.
\end{theo}

For nonsoluble finite groups we prove similar results in terms of the nonsoluble length and the generalized Fitting height of $R_{n}(g)$. We recall the relevant definitions. The \textit{generalized Fitting series} is defined recursively by $F_{0}^{*}(G)=1$ and $F_{i+1}^{*}(G)/F_{i}^{*}(G)=F^{*}(G/F_{i}^{*}(G))$, where $F^{*}(H)$ is the generalized Fitting subgroup of $H$, which is the characteristic subgroup generated by the Fitting subgroup $F(H)$ and all the subnormal quasisimple subgroups (a group $Q$ is quasisimple if $[Q,Q]=Q$ and $Q/Z(Q)$ is simple). The least number $h$ such that $F_{h}^{*}(G)=G$ is called the
\textit{generalized Fitting height} $h^{*}(G)$ of $G$. Notice that if $G$ is soluble, then $h^{*}(G)=h(G)$ is the ordinary Fitting height of $G$. We are now ready for the statement of the second main result.

\begin{theo}\label{t1.2}
Let $m$ and $n$ be positive integers, and let $g$ be an element of a finite group $G$ whose order $|g|$ is the product of $m$ primes counting multiplicities. If the generalized Fitting height of $R_{n}(g)$ is equal to $k$, then $g$ belongs to $F_{f(k,m)}^{*}(G)$ where $$f(k,m)=((k+1)m(m+1)+2)(k+3)/2.$$
\end{theo}

The proof of Theorem~\ref{t1.2}  follows from Theorem~\ref{t1.1} and our last result which involves another length parameter, the nonsoluble length. We define the \textit{upper nonsoluble series} of a finite group $G$ recursively as follows. Let $B_{0}=1$ and let $D_{0}$ be the soluble radical of $G$. Then let $B_{i+1}/D_{i}$ be the product of all the normal simple non-abelian subgroups of $G/D_{i}$, and $D_{i+1}/B_{i+1}$ the soluble radical of $G/B_{i+1}$. The \textit{nonsoluble length}  $\lambda(G)$ is defined to be the first integer $l$ such that $D_{l}=G$, so that we have a normal series
$$
1=B_{0}\leq D_{0}<B_{1}\leq D_{1}\cdots <B_{l}\leq D_{l}=G.
$$
It is not difficult to see that $\lambda(G)$ is the minimum number of nonsoluble factors in a normal series where each factor is either soluble or a direct product of non-abelian simple groups. In particular, a group is soluble if and only if the nonsoluble length is $0$. We can now state our final main result.

\begin{theo}\label{t1.3}
Let $m$ and $n$ be positive integers, and let $g$ be an element of a finite group $G$ whose order $|g|$ is
the product of $m$ primes counting multiplicities. If the nonsoluble length of $R_{n}(g)$ is equal to $k$, then
$g$ belongs to $D_{f_1(k,m)}(G)$ where $f_1(k,m)=(k+1)m(m+1)/2$.
\end{theo}

Theorems~\ref{t1.2} and \ref{t1.3} depend on the classification of finite simple groups in so far as they rely on the validity of the Schreier conjecture that
the outer automorphism group of a finite simple group is soluble. Bounds for the nonsoluble length and generalized Fitting height of a finite group were implicitly used in the reduction of the Restricted Burnside Problem to soluble and nilpotent groups in the Hall--Higman paper \cite{Ha}. Such bounds also find applications in the study of profinite groups; see for example \cite{Kha,W}.

It is well-known that the inverse of a right Engel element is a left Engel element.
But there is no straightforward connection between the right and left Engel-type subgroups $R_n(G)$ and $E_n(G)$. Therefore the results of the current paper do not follow directly from the results on the subgroups $E_n(G)$ in \cite{Khb}, although the proofs make use of several auxiliary propositions from that paper.

\section{Preliminaries}

First we mention some elementary radical properties of the subgroups $F_{i}(G)$, $F_{i}^{*}(G)$, and $D_{i}(G)$ that will be used throughout the paper.
If $N$ is a normal subgroup of $G$, then
 $$
 F_{i}(N)=N\cap F_{i}(G),\quad F_{i}^{*}(N)=N\cap F_{i}^{*}(G),\quad
		 D_{i}(N)=N\cap D_{i}(G).
$$
It follows that the same equalities hold whenever $N$ is a subnormal subgroup of~$G$. Therefore the Fitting height (when $G$ is soluble), the generalized Fitting height, and the nonsoluble length of $N$ do not exceed the corresponding parameters of $G$ when $N$ is subnormal in $G$.

Let $S_{i}(G)$ be one of $F_{i}(G)$, $F_{i}^{*}(G)$, or $D_{i}(G)$. If $N$ is a subnormal subgroup of $G$, then for the normal closures in $G$ we have
$$
\langle S_{i}(N)^{G}\rangle =\langle (N\cap S_{i}(G))^{G}\rangle \leq \langle N^{G}\rangle \cap S_{i}(G)=S_{i}(\langle N^{G}\rangle ).
$$
Since $N$ is subnormal in $\langle N^{G}\rangle$, we also have $S_{i}(N)=N\cap S_{i}(\langle N^{G}\rangle)$. Hence the Fitting height (when $G$ is soluble), the generalized Fitting height,  and the nonsoluble length of $N$ are the same as for $\langle N^{G}\rangle$.

We collect these properties in the following lemma for further references.

\begin{lemm}
\label{l-sn}
Let $N$ be a subnormal subgroup of a finite group $G$.

{\rm (a)} The Fitting height (when $G$ is soluble), the generalized Fitting height, and the nonsoluble length of $N$ do not exceed the corresponding parameters of $G$.

{\rm (b)} The Fitting height (when $N$ is soluble), the generalized Fitting height, and the nonsoluble length of the normal closure $\langle N^G\rangle$ are equal to the corresponding parameters of~$N$.
\end{lemm}

Furthermore, $F_{i}(G)\cap H\leq F_{i}(H)$ whenever $H\leq G$, and for any normal subgroup $N$ of $G$ we have
$$
F_{i}(G/N)\geq F_{i}(G)N/N,\quad F_{i}^{*}(G/N)\geq F_{i}^{*}(G)N/N,\quad  D_{i}(G/N)\geq D_{i}(G)N/N.
$$

When a group $A$ acts by automorphisms on a  group $G$, we regard $A$ as a subgroup of the semidirect product $GA$ and use the commutator and centralizer notation like $[G,A]$ and $C_G(A)$. We use without special references the well-known properties of coprime actions:  if $\alpha $ is an automorphism of a finite group $G$ of coprime order, $(|\alpha |,|G|)=1$, then $ C_{G/N}(\alpha )=C_G(\alpha )N/N$ for any $\alpha $-invariant normal subgroup $N$, the equality  $[G,\alpha ]=[[G,\alpha,],\alpha ]$ holds, and if $G$ is in addition abelian, then $G=[G,\alpha ]\times C_G(\alpha )$.

It is convenient to introduce a more specialized notation for right Engel-type subgroups. Let $G$ be a group and $g\in G$. For a $g$-invariant subgroup $H$ of $G$ we set
$$
R_{H,n}(g)=\langle [g,{}_{n}x]:\,x\in H\rangle.
$$
Thus, for $H=G$ this is the subgroup $R_{n}(g)=\langle [g,{}_{n}x]:\,x\in G\rangle$ introduced above. It is clear from the definition that for $g$-invariant subgroups $K\leq H$ we have $R_{K,n}(g)\leq R_{H,n}(g)$, and if $N$ is a $g$-invariant normal subgroup of $H$, then $R_{HN/N,n}(\bar g)$ is the image of $R_{H,n}(g)$ in $H/N$. Another obvious property is that $C_{H}(g)$ normalizes $R_{H,n}(g)$. We shall use these properties of right Engel-type subgroups without special references.

We record here a simple property of abelian-by-cyclic subgroups.

\begin{lemm}\label{l0}
If $A$ is an abelian subgroup of a group $G$ normalized by an element $g\in G$, then for any $a\in A$ we have $[g,{}_n ga]=[a^{-1},{}_ng]$.
\end{lemm}

\bp
We have $[g,ga]=[ga,g]^{-1}=[a,g]^{-1}=[a^{-1},g]$ since $A$ is abelian. For the same reason, the action of $ga$ on $A$ is the same as that of $g$, whence the result.
\ep

We shall also need a couple of lemmas on direct products of nonabelian finite simple groups; the first lemma was proved in \cite{Khb}.

\begin{lemm}[{\cite[Lemma~3.8]{Khb}}]\label{l2.2} Let $S=S_{1}\times \cdots \times S_{p}$ be a direct product of $p$ isomorphic  non-abelian finite simple groups, where $p$ is a prime, and let $\varphi$ be the natural automorphism of
$S$ of order $p$ that regularly permutes the $S_{i}$. Let $n$ be a positive integer,  and let
$$
F=\langle [x,{}_{n}\varphi]:\, x\in S_{i},\;i=1,2,\dots,r\rangle.
$$
Then $F=S$.
\end{lemm}

As a corollary we get the following.

\begin{lemm}\label{l2.3} Let $S=S_{1}\times \cdots \times S_{r}$ be a direct product of $r$ isomorphic  non-abelian finite simple groups, and let $\varphi$ be the natural automorphism of $S$ of order $r$ that regularly permutes the $S_{i}$. Let $n$ be a positive integer. Then $S=R_{S\langle \varphi\rangle,n}(\varphi)$.
\end{lemm}

\begin{proof} Let $\varphi^{k}$ be of prime order (possibly, with $k=1$). By Lemma~\ref{l2.2} applied to each orbit of $\varphi^{k}$, we have
\begin{equation}\label{e0}
S=\langle [g,{}_{n}\varphi^{k}]:\,g\in S_{i},\;i=1,2,\ldots ,r\rangle.
\end{equation}
For any $1\leq i\leq r$ and any $g\in S_{i}$, the subgroup
$A(g)=\langle g^{\langle \varphi\rangle}\rangle$ is abelian and $\varphi$-invariant. Using this fact and the commutator formulae $[b,\varphi^{i+1}]=[b,\varphi^{i}][b^{\varphi^{i}},\varphi]$ and $[bc,\varphi^{i}]=[b,\varphi^{i}][c,\varphi^{i}]$ that hold for all $b,c\in A(g)$, we see that any commutator $[g,{}_{n}\varphi^{k}]$ belongs to the subgroup  generated by commutators of the form $[x,{}_{n}\varphi]$ with
$x\in A(g)$. By Lemma~\ref{l0} we also have
$[x,{}_{n}\varphi]=[\varphi, {}_n\varphi x^{-1}]\in R_{S\langle \varphi\rangle,n}(\varphi)$. Combined with \eqref{e0}, this proves that $S=R_{S\langle \varphi\rangle,n}(\varphi)$.
\end{proof}

\section{Fitting height}

In this section we prove Theorem~\ref{t1.1} on finite soluble groups.

\begin{prop}\label{p2.1}
Let $\alpha $ be an automorphism of a finite soluble group $G$ such that $G=[G,\alpha]$.  Then $G=R_{G\langle\alpha\rangle,n}(\alpha)$
for all $n\geq 1$.
\end{prop}

\begin{proof} Let $R=R_{G\langle\alpha\rangle,n}(\alpha)$. We argue by contradiction and let $G$ be a counterexample of minimal order. Let $M$ be a minimal normal $\alpha$-invariant subgroup of $G$ (and thus an elementary abelian $p$-group for some prime $p$). By the minimality of $G$ we have $G/M=R_{G\langle\alpha\rangle/M,n}(\bar{\alpha})=MR/M$ (where $\bar {\alpha}$ is the automorphism of $G/M$ induced by $\alpha$), and thus $G=MR$. Since $G\not =R$, we
must have $M\not\leq R$ and therefore, $M\cap R\ne M$. Since $M$ is abelian and normal, it follows that $M\cap R$ is normal in $MR=G$. Then by the minimality of $M$ we must have $M\cap R=1$. Thus,
$$
G=MR\qquad \text{and}\qquad M\cap R=1.
$$

We claim that $C_{M}(\alpha)=1$. If this is not the case, then, as $C_{M}(\alpha)$ normalizes $R$, we have $[C_{M}(\alpha),R]\leq M\cap R=1$, so that  $C_{M}(\alpha)$ centralizes $R$ as well as $M$. Hence  $C_{M}(\alpha)$ is an $\alpha$-invariant subgroup contained in the centre of $G=MR$, and by the minimality of $M$ it follows that
$C_{M}(\alpha)=M$. But then $G=[G,\alpha ]=[R,\alpha]\leq R$, a contradiction.

Thus, $C_{M}(\alpha)=1$, which implies that the mapping $M\rightarrow M$ given by $m\mapsto [m,\alpha]$ is injective. Since $M$ is
finite, it follows that this mapping is surjective, and therefore every element in $M$ is of the form $[m,{}_{n}\alpha]$ for
some $m\in M$. Hence, using Lemma~\ref{l0}, we obtain
$$
 M=\langle [m,{}_{n}\alpha]:\,m\in M\rangle =\langle [\alpha,{}_{n}\alpha m]:\,m\in M\rangle \leq R.
$$
Then $G=MR=R$, a contradiction.
\end{proof}

\begin{proof}[Proof of Theorem~\ref{t1.1}]
Let $g$ be an element of a finite soluble group $G$, and $n$ a positive integer. Suppose that the Fitting height of
$R_{n}(g)$ is equal to $k$. We want to show that $g$ then belongs to $F_{k+1}(G)$.

Consider the subnormal series
$$
G\geq [G,g]\geq [[G,g],g]\geq \cdots \geq [...[[G,\underbrace{g],g],\dots ,g}_{m}]=H,
$$
where $[H,g]=H$ (can also be trivial). Let $N=\langle H^{G}\rangle$ be the normal closure of $H$. Since $gN$ is a left Engel element in $G/N$, we know from Baer's theorem that   $gN\in F(G/N)$. Hence it
suffices to show that $N$ is of Fitting height at most $k$. Since $H$ is subnormal in $G$, by Lemma~\ref{l-sn}(b) it suffices to show
that $H$ has Fitting height at most $k$. Since $H=[H,g]$, we know  from Proposition~\ref{p2.1} that $H=R_{H\langle g\rangle,n}(g)$. We have $R_{H\langle g\rangle,n}(g)\leq  R_{n}(g)$, and  by hypothesis $R_{n}(g)$ has Fitting height $k$. Therefore $H$ has Fitting height at most $k$, as required.
\end{proof}

\section{Nonsoluble length}

In this section we prove Theorem~\ref{t1.3} on nonsoluble length.
First we introduce some notation. Suppose that the nonsoluble length
of a finite group $G$ is $\lambda(G)=\lambda$. This means that we
have the upper nonsoluble series
\begin{equation}\label{e1}
1=B_{0}\leq D_{0}<B_{1}\leq D_{1}\cdots
	<B_{\lambda}\leq D_{\lambda}=G.
\end{equation}
Recall that each factor $D_{j}/B_{j}$ is the soluble radical (possibly, trivial) of $G/B_{j}$. We fix the notation $U_{j}=B_{j}/D_{j-1}$ for the other factors, each of which is the (nontrivial) direct product of all subnormal nonabelian simple subgroups of $G/D_{j-1}$, say,  $U_j=S_{1}\times \dots \times S_{v}$ (to lighten the notation, we do not use double indices for $S_i$ here, but the groups and the number of factors may of course be different for different $j$). Acting by conjugation the group $G$ permutes these subnormal factors $S_i$; for brevity we speak of \textit{orbits of elements of $G$ on $U_j$} meaning orbits in this permutational action. The stabilizer of a point $S_i$ can also be denoted as the normalizer $N_G(S_i)$ of the section $S_i$. Let $K_{j}$ be the kernel of the permutational action of $G$ on $\{S_{1},\ldots ,S_{v}\}$. Clearly, $B_{j}\leq K_{j}$, and since $K_{j}/B_{j}$ is soluble by the Schreier conjecture, we have $K_{j}\leq D_{j}$.

The following proposition is the main tool in the proof of Theorem~\ref{t1.3}; it is a right Engel analogue of Proposition~5.4 in \cite{Khb}.

\begin{prop}\label{p2.4}
Let $g$ be an element of a finite group $G$ whose order $|g|$ is a product of $m>1$ primes. Suppose that $\langle g\rangle\cap K_{ms}=1$ for some positive integer $s$. Then for any positive integer $n$ the nonsoluble length of $R_{n}(g)$ is at least $s$.
\end{prop}

Before embarking on the proof of this proposition, we introduce some definitions.
Let $g$ be an element of $G$, and let $\{ S_{{1}} ,\ldots ,S_{{r}}\}$ be a $g$-orbit on $U_i$, that is, on the set of simple factors the product of which is $U_i$. As in \cite{Khb},  we
say that such an orbit is \textit{pure} if the order of the automorphism of $S=S_{{1}}\times\cdots\times S_{{r}}$
induced by $g$ acting by conjugation is equal to $r$, which is the order of the permutation induced by $g$ on this orbit.
 In other words, the orbit is pure if the stabilizer of
a point in $\langle g\rangle$ acts trivially on $S$, that is,  $N_{\langle g\rangle}(S_{1})=C_{\langle g\rangle}(S)$. We introduce the following hypothesis, which is the same as Hypothesis~5.1 in \cite{Khb}.

\begin{hypo}\label{h3.1}
Let $g$ be an element of a finite group $G$, and let $\{S_{1},\ldots ,S_{r}\}$ be a $g$-orbit in a section $U_{i}=B_{i}/D_{i-1}$ of the series \eqref{e1} with $i\geq 2$. Suppose that $\langle g\rangle\cap B_{i}=1$, and that $g$ acting by conjugation induces a nontrivial automorphism $\bar{g}$ of $S=S_{1}\cdots S_{r}$. Let $t=|\bar{g}|$, so that $\langle g^{t}\rangle$ is the
centralizer of $S$ in $\langle g\rangle$ and $\langle \bar{g}\rangle =\langle g\rangle/\langle g^{t}\rangle$.

Furthermore, let $\hat{S}$ be a minimal by inclusion $g$-invariant subgroup of $G$ such that $\hat{S}D_{i-1}/D_{i-1}=S$. If
in addition $\{S_{1},\ldots ,S_{r}\}$ is a pure $g$-orbit, then we choose $\hat{S}$ to be contained in $R_{n}(g)$, which is possible by Lemma~\ref{l2.3}. Let $\hat{S}_{1}$ be a minimal by inclusion subgroup of $\hat{S}$ such that $\hat{S}_{1}D_{i-1}/D_{i-1}=S_{1}$.
\end{hypo}

We use two technical propositions from \cite{Khb}. The first one provides a passage from a pure orbit of $g$ on $U_{i}$ to an orbit of at least the same length on the preceding section $U_{i-1}$.

\begin{prop}[{\cite[Proposition 5.2]{Khb}}]\label{p3.2}
Assume Hypothesis~\ref{h3.1}. If $\{S_{1},\ldots ,S_{r}\}$ is a pure $g$-orbit, then $U_{i-1}$ contains a $g$-orbit $\{T_{1},\ldots ,T_{l}\}$ of length divisible by $r=|\bar{g}|$, and for some choice of $\hat{S}_{1}$ there is an element $x_{0}\in \hat{S}_{1}\setminus D_{i-1}$ that does not belong to $N_{G}(T_{1})$.
\end{prop}

The second technical proposition from \cite{Khb} provides a passage from a non-pure orbit of $g$ in $U_{i}$ to an orbit in $U_{i-1}$ that is strictly greater with respect to the following ordering. Namely, on the set of positive integers we introduce the lexicographical order with respect to the exponents of primes in the canonical prime-power decomposition: if $a=2^{i_{2}}3^{i_{3}}5^{i_{5}}\cdots $ and $b=2^{j_{2}}3^{j_{3}}5^{j_{5}}\cdots$, then by definition $a\prec b$ if, for some prime $p$, we have $i_{q}=j_{q}$ for
all primes $q<p$ and $i_{p}<j_{p}$. In particular,  if $v$ is a proper divisor of $u$,  then $v\prec u$.

\begin{prop}[{\cite[Proposition 5.3]{Khb}}]\label{p3.3} Assume Hypothesis 3.1. If $\{S_{1},\ldots ,S_{r}\}$ is a non-pure $g$-orbit, then $U_{i-1}$ contains a $g$-orbit $\{T_{1},\ldots ,T_{l}\}$ of length strictly greater than $r$ with respect to the order $\prec$.
\end{prop}

\begin{proof}[Proof of Proposition~\ref{p2.4}] We have an element $g\in G$ of order equal to a product of $m$ primes such that
$\langle g\rangle\cap K_{ms}=1$. (Recall that $K_{j}$ is the kernel of the permutational action of $G$ on the simple factors whose product is $U_j$.) We want to show that the nonsoluble length of $R_{n}(g)$ is at least $s$. With this in mind, we consider the upper nonsoluble series for $R_{n}(g)$ constructed in the same way as \eqref{e1} was constructed for $G$, with its terms denoted by \begin{equation}\label{e2}
 1=\beta_{0}\leq \delta_{0}<\beta_{1}\leq \delta_{1}\cdots <\beta_{e}\leq \delta_{e}=R_{n}(g).
\end{equation}
Here each factor $\delta_{j}/\beta_{j}$ is the soluble radical (possibly, trivial) of $G/\beta_{j}$, and each factor $\beta_{j}/\delta_{j-1}$ is the (nontrivial) direct product of all subnormal nonabelian simple subgroups of $G/\delta _{j-1}$. Our task is to show that $e\geq s$.

Since $\langle g\rangle\cap K_{ms}=1$, the element $g$ has at least one nontrivial orbit on the
set of simple factors of $U_{ms}=B_{ms}/D_{ms-1}$, say, $\{S_{1},\ldots ,S_{r}\}$. If the orbit is pure, then we apply Proposition~\ref{p3.2} to this orbit. If the orbit is not pure, then we apply Proposition~\ref{p3.3}. Then we apply the same procedure to the
orbit $\{T_{1},\ldots ,T_{l}\}$ in $U_{ms-1}$ just obtained: this orbit takes the role of the orbit $\{S_{1},\ldots ,S_{r}\}$ in
Proposition~\ref{p3.2} or \ref{p3.3} depending on whether it is pure or not. We proceed with constructing this sequence of orbits, descending over
the sections $U_{i}$ making $ms-1$ such steps. If we make such a step from a pure orbit by Proposition~\ref{p3.2}, then the length of the new orbit is divisible by the length of the old orbit and therefore does not decrease with respect to the order $\prec$. If we make such a
step from a non-pure orbit by Proposition~\ref{p3.3}, then the length of the new orbit is strictly greater with respect to the order $\prec$ than the length of the old orbit. In the sequence of orbits thus constructed, some orbits may be pure, some not. Notice that there can be at most $m-1$ passages from a non-pure orbit, since $|g|$ is a product of $m$ primes. As a result, if the length of the sequence is at least $(s-1)m+(m-1)+1=ms$, then
this sequence must necessarily contain $s$ successive pure orbits, with $s-1$ consecutive passages between them. Let $i=t,t-1,\ldots ,
t-s+1$ be the indices of the corresponding sections $U_{i}$.

Recall that if a $g$-orbit $\{S_{1},\ldots ,S_{r}\}$ in $U_{i}$ is pure, then by Lemma~\ref{l2.3} the subgroup $S=S_{1}\cdots S_{r}$ is contained in the image of $R_{n}(g)$ in $G/D_{i-1}$ (since this image is obviously equal to the analogous subgroup $R_{n}(g)$ constructed for $G/D_{i-1}$).
The idea is to use each of these $s$ consecutive pure orbits in $U_{t},U_{t-1},\ldots ,U_{r-s+1}$ to `mark' a nonsoluble factor of the series \eqref{e2} and prove that the factor marked by the pure orbit in $U_{i-1}$ is necessarily `lower' in \eqref{e2} than the factor marked by
the pure orbit in $U_{i}$, for every $i=t,t-1,\ldots ,t-s+2$. Then the series \eqref{e2} must realize nonsoluble length at least $s$.

Thus, let $\{S_{1},\ldots ,S_{r}\}$ be a pure $g$-orbit in $U_{i}$, and $\{T_{1},\ldots ,T_{l}\}$ the pure $g$-orbit in $U_{i-1}$ obtained by Proposition~\ref{p3.2}. Recall that in accordance with Hypothesis 3.1, $\hat{S}$ is a minimal by inclusion $g$-invariant subgroup
of $R_{n}(g)$ such that $S=\hat{S}D_{i-1}/D_{i-1}$, and $\hat{S}_{1}$ is a minimal by inclusion subgroup of $\hat{S}$ such
that $\hat{S}_{1}D_{i-1}/D_{i-1}=S_{1}$. Note that since $S_{1}$ is nonabelian simple, $\hat{S}_{1}$ has no nontrivial soluble homomorphic images:
\begin{equation}\label{e3}
 \hat{S}_{1}=[\hat{S}_{1},\hat{S}_{1}].
\end{equation}
By Proposition~\ref{p3.2} we have an element $x_{0}\in \hat{S}_{1}\setminus D_{i-1}$ such that
\begin{equation}\label{e4}
 T_{1}^{x_{0}}\not =T_{1}.
\end{equation}
Consider the image of the series \eqref{e2} in $R_{n}(g)D_{i-1}/D_{i-1}$. Since $\hat{S}_{1}D_{i-1}/D_{i-1}\cong S_{1}$
is a subnormal nonabelian simple subgroup of $R_{n}(g)D_{i-1}/D_{i-1}$, we obviously have a well-defined index $j$ such that
\begin{equation}\label{e5}
 \hat{S}_{1}\leq \beta_{j}D_{i-1}\qquad \text{and}\qquad \hat{S}_{1}\not\leq \delta_{j-1}D_{i-1}.
\end{equation}
Notice that then also
\begin{equation}\label{e6}
 \hat{S}_{1}\leq \beta_{j}(D_{i-1}\cap R_{n}(g)).
\end{equation}
It is clear that the index $j$ depends only on $S_{1}$ (that is, it is independent of the choice of $\hat{S}_{1}$ such that
$\hat{S}_{1}D_{i-1}/D_{i-1}=S_{1}$). Since $\{T_{1},\ldots ,T_{l}\}$ is  a pure $g$-orbit in $U_{i-1}$, the product
$T=T_{1}\cdots T_{l}$ is also covered by $R_{n}(g)$ by Lemma~\ref{l2.3}. If $\hat{T}_{1}$ is any subgroup of $R_{n}(g)$ such that $\hat{T}_{1}D_{i-2}/D_{i-2}=T_{1}$, then again, there is a well-defined index $u$ depending only on $T_{1}$ such that
\begin{equation}\label{e7}
\hat{T}_{1}\leq \beta_{u}(D_{i-2}\cap R_{n}(g))\qquad \text{and} \qquad \hat{T}_{1}\not\leq \delta_{u-1}D_{i-2}.
\end{equation}

We will need the following lemma to finish the proof of Proposition~\ref{p2.4}.

\begin{lemm}\label{l3.4}
Under the above hypotheses, $j>u$.
\end{lemm}

\begin{proof}
We argue by contradiction and suppose that $j\leq u$. Then \eqref{e6} implies that $\hat{S}_{1}\leq \beta_{u}(D_{i-1}\cap
R_{n}(g))$. The image of $\beta_{u}/\delta_{u-1}$ in $R_{n}(g)/(D_{i-2}\cap R_{n}(g))$ is a direct product of nonabelian simple groups, one of which is $\bar{T}_{1}=\hat{T}_{1}/(\hat{T}_{1}\cap D_{i-2})\cong T_{1}$ by \eqref{e7}. Acting by conjugation the group $R_{n}(g)$ permutes these factors. Consider the permutational action of $R_{n}(g)$ on the orbit containing $\bar{T}_{1}$. Clearly, $\beta_{u}$ is contained in the kernel of this action. The subgroup $B_{i-1}\cap R_{n}(g)$ normalizes $\hat{T}_{1}$ modulo $D_{i-2}\cap R_{n}(g)$ (since $B_{i-1}$ normalizes $\hat{T}_{1}$ modulo $D_{i-2}$). As a normal subgroup of $R_{n}(g)$, then $B_{i-1}\cap R_{n}(g)$ is contained in the kernel of that action on the orbit containing $\hat{T}_{1}$. Since $D_{i-1}/B_{i-1}$ is soluble, we obtain that the image of $\hat{S}_{1}\leq \beta_{u}(D_{i-1}\cap R_{n}(g))$ in this action is soluble. But $\hat{S}_{1}=[\hat{S}_{1},\hat{S}_{1}]$ by minimality of $\hat{S}_{1}$, as noted in \eqref{e3}. Therefore this image must actually be trivial; in particular,
\begin{equation}\label{e8}
 \hat{T}_{1}^{x_{0}}\equiv \hat{T}_{1}\;(\mbox{mod\,}\delta_{u-1}(D_{i-2}\cap R_{n}(g))).
\end{equation}
On the other hand, since $T_{1}^{x_{0}}\not =T_{1}$ by \eqref{e4}, it follows that $[\hat{T}_{1},\hat{T}_{1}^{x_{0}}]\leq D_{i-2}\cap R_{n}(g)$. Together with \eqref{e8} this implies
$$
\bar{T}_{1}=[\bar{T}_{1},\bar{T}_{1}]\equiv 1\;(\mbox{mod\,}\delta_{u-1}(D_{i-2}\cap R_{n}(g))),
$$
contrary to \eqref{e7}. Lemma~\ref{l3.4} is proved.
\end{proof}

We now finish the proof of Proposition~\ref{p2.4}. In our sequence of orbits constructed above, each pure orbit $\{S_{1},\ldots ,S_{r}\}$ in $U_{i}$ for $i=t,\ldots ,t-s+1$
marks a nonsoluble quotient $\beta_{j}/\delta_{j-1}$ of the series \eqref{e2} in the sense of \eqref{e5}. By Lemma~\ref{l3.4} the next pure orbit in $U_{i-1}$ marks a strictly lower section. Therefore there must be at least $s$ different nonsoluble sections in \eqref{e2}, since their indices must be strictly descending as we go over the $s$ consecutive pure orbits. As a result, the nonsoluble length of $R_{n}(g)$ is at least $s$.
\end{proof}

We are now ready to finish the proof of Theorem~\ref{t1.3}.

\begin{proof}[Proof of Theorem~\ref{t1.3}]
Recall that we have $\lambda(R_{n}(g))=k$ and  $|g|$ is a product of $m$ primes counting multiplicities. We want to show that $g\in
D_{(k+1)m(m+1)/2}(G)$. We proceed by induction on $m$. As a base of this induction we can take the case of $m=0$, when $g=1$ and certainly $1\in D_{(k+1)\cdot 0\cdot (0+1)/2}(G)=D_0(G)$.

For $m>0$, by Proposition~\ref{p2.4} we have
$$
\langle g\rangle \cap D_{(k+1)m}\geq \langle g\rangle \cap K_{(k+1)m}\not =1.
$$
The corresponding subgroup $R_{n}(\bar{g})$ constructed for the image $\bar{g}$ of $g$ in $\bar{G}=G/D_{(k+1)m}$ is clearly a quotient of $R_{n}(g)$ and therefore has nonsoluble length at most $k$. As $|\bar{g}|$ is a product of at most $m-1$ primes, we have  $\bar{g}\in D_{(k+1)(m-1)m/2}(\bar{G})$ by the induction hypothesis. The result now
follows, since $D_{(k+1)(m-1)m/2}(\bar{G})$ is the image of $D_{(k+1)m(m+1)/2}$ in $G/D_{(k+1)m}$.
\end{proof}

\section{Generalized Fitting height}

In this section we prove Theorem~\ref{t1.2} on generalized Fitting height $h^*(G)$. This is achieved by combination of Theorem~\ref{t1.3} on nonsoluble length and Theorem~\ref{t1.1} on Fitting height.

\begin{proof}[Proof of Theorem~\ref{t1.2}]
Recall that $g$ is an element of a finite group $G$ whose order is a product of $m$ primes such that the generalized Fitting height of the right Engel sink $R_{n}(g)$ is equal to $k$. We need to show that $g$ belongs to
$F^{*}_{((k+1)m(m+1)+2)(k+3)/2}(G)$. As in the proof of Theorem~\ref{t1.1},  we consider the subnormal series
$$
G\geq [G,g]\geq [[G,g],g]\geq \cdots \geq [...[[G,\underbrace{g],g],\dots ,g}_{m}]=H,
$$
where $[H,g]=H$ (can also be trivial). Let $N=\langle H^{G}\rangle$ be the normal closure of $H$. Since $gN$ is a left Engel element in $G/N$, we know from Baer's theorem that   $gN\in F(G/N)$. Hence it
suffices to show that the generalized Fitting height of $N$ is at most  $((k+1)m(m+1)+2)(k+3)/2-1$.  Since $H$ is subnormal in $G$, by Lemma~\ref{l-sn}(b) it suffices to obtain the given bound for the generalized Fitting height of $H$.

 Since the nonsoluble
length of $R_{n}(g)$ does not exceed its generalized Fitting height, we know from Theorem~\ref{t1.3} that $g$ belongs to $D_{(k+1)m(m+1)/2}(G)$. Since $H=[H,g]$, it follows that
$H\leq D_{(k+1)m(m+1)/2}(G)$. Since $H$ is a subnormal subgroup, it follows that $\lambda(H)\leq (k+1)m(m+1)/2$ by Lemma~\ref{l-sn}(a). It remains to obtain the necessary bounds for the Fitting
height of the $(k+1)m(m+1)/2+1$ soluble factors of the upper
nonsoluble series of $H$, which are $g$-invariant (because $H=[H,g]$ is $g$-invariant).

Let $X=Y/Z$ be one of those soluble factors, where $Y$ and $Z$ are
$g$-invariant and normal in $H$ and therefore subnormal in $G$. Since $Y$ is subnormal, we have $h^{*}(R_{n}(g)\cap Y)\leq
h^{*}(R_{n}(g))=k$ by Lemma~\ref{l-sn}(a). Then the  image
of $R_{n}(g)\cap Y$ in $Y/Z$, which is soluble, has Fitting height at most $k$. Since, clearly, $R_{Y\langle g\rangle,n}(g)\leq R_{n}(g)\cap Y$, we obtain that $R_{X\langle g\rangle,n}(g)$, which is the
image of $R_{Y\langle g\rangle,n}(g)$, also has Fitting height
at most $k$.

Consider the semidirect product $X\rtimes \langle g\rangle$. By applying Theorem~\ref{t1.1} we obtain that
$g\in F_{k+1}(X\rtimes \langle g\rangle )$. Therefore,
$[X,g]\leq F_{k+1}(X\rtimes \langle g\rangle )\cap X\leq
F_{k+1}(X)$. In other words, $g$ acts trivially on $X/F_{k+1}(X)$. Since $H=[H,g]$, it follows that $H$ also acts trivially on $X/F_{k+1}(X)$, that is, $X/F_{k+1}(X)$ is a central section of $H$. In particular, $X=F_{k+2}(X)$ and thus the Fitting height of $X$ is at most $k+2$. Hence the generalized Fitting height of $H$ is at most
 $$
 ((k+1)m(m+1)/2+1)(k+2)+(k+1)m(m+1)/2=((k+1)m(m+1)+2)(k+3)/2-1,
 $$
as required.
\end{proof}

\section*{Acknowledgements}
The first author was supported  by the Russian Science Foundation, project no. 14-21-00065, the second by FAPDF, Brazil, and the third by EPSRC.
The bulk of this work was done during the visits of the first and third authors  to the University of Brasilia and they are grateful for the great hospitality of the university during these visits, which were supported by CNPq-Brazil for the first author, and  partly supported by EPSRC for the third author.


\begin{thebibliography}{99}

\bibitem{Ha} P. Hall and G. Higman, On the $p$-length of $p$-soluble groups and reduction theorems for Burnsides's problem, \textit{Proc. London Math. Soc.} (3) {\bf 6} (1956), 1--42.

\bibitem{Ba} R. Baer, Engelsche Elemente Noetherscher Gruppen, \textit{Math. Ann.} {\bf 133} (1957), 256--270.

\bibitem{Kha} E. I. Khukhro and P. Shumyatsky, Words and pronilpotent subgroups in profinite groups, \textit{J. Austral. Math. Soc.} {\bf 97} (2014), 343--364.

\bibitem{Khb} E. I. Khukhro and P. Shumyatsky, Engel-type subgroups and length parameters of finite groups, \textit{Israel J. Math.} {\bf 222} (2017), 599--629.

\bibitem{W} J. S. Wilson, On the structure of compact torsion groups, \textit{Monatsh. Math.} {\bf 96} (1983), 57--66.

\end{thebibliography}
\end{document}